\documentclass[12pt,a4paper, twoside, reqno]{amsart}

\usepackage[utf8]{inputenc}
\usepackage[T1]{fontenc}
\usepackage{amsmath,amsfonts,amssymb,amsthm,bbm,mathrsfs,}
\usepackage{longtable}
\usepackage{layout}
\usepackage{mflogo}
\usepackage{xcolor}
\usepackage{hyperref}
\hypersetup{
    colorlinks=true,
    linkcolor=blue,
    filecolor=magenta,      
    urlcolor=cyan,
    pdftitle={Overleaf Example},
    pdfpagemode=FullScreen,
    }
\usepackage{graphicx}
\usepackage{cleveref}

\usepackage{setspace}

\usepackage{fullpage}
\newcommand{\C}{\mathbbm{C}}

\newcommand{\Q}{\mathbbm{Q}}             
\newcommand{\R}{\mathbbm{R}}             

\newcommand{\sym}{\mathrm{sym}}

\theoremstyle{definition}
\newtheorem{definition}{Definition}[section]
\theoremstyle{plain}
\newtheorem{lem}[definition]{Lemma}
\newtheorem{thm}[definition]{Theorem}

\newtheorem{conj}[definition]{Conjecture}

\theoremstyle{remark} 
\newtheorem{rem}[definition]{Remark}

\title{Joint distribution of Hecke eigenforms on $ \mathbbm{H} ^3$}  

\author{Didier Lesesvre}
\address{CNRS – Université de Montréal CRM – CNRS \newline 
University of Lille --- CNRS, UMR 8524 --- Laboratoire Paul Painlevé \newline
59000 Lille, France}
\email{didier.lesesvre@univ-lille.fr}

\author{Luca Marchesini}
\address{Sorbonne Université \newline CNRS, UMR 7586 --- Institut Mathématique de Jussieu \newline
4 Place Jussieu, 75252, Paris, France }
\address{Scuola Galileiana di Studi Superiori di Padova \newline Università di Padova \newline Via Venezia 20, 35131 Padova, Italia}
\email{luca.marchesini@imj-prg.fr}

\author{Nicole Raulf}
\address{University of Lille \newline CNRS, UMR 8524 --- Laboratoire Paul Painlevé \newline
59000 Lille, France}
\email{nicole.raulf@univ-lille.fr}

\date{\today} 

\begin{document}

\maketitle

\begin{abstract}
We prove a joint value equidistribution statement for Hecke-Maa{\ss} cusp forms on the hyperbolic three-space $\mathbb{H}^3$. This supports the conjectural statistical independence of orthogonal cusp forms.
\end{abstract}

\section{Introduction}

\subsection{Value distribution conjecture}

The value distribution of eigenfunctions in the semiclassical limit is one of the main problems in analytic number theory and quantum chaos. Berry \cite{berry_regular_1977} suggested that eigenfunctions of the Laplace operator for chaotic systems are modeled by random waves: in particular, eigenfunctions on a compact hyperbolic surface should have a Gaussian value distribution when the frequency grows, and the moments of an $L^2$-normalized eigenfunction should be given by the Gaussian moments. Such a behavior is analogously expected from automorphic forms on more general locally symmetric spaces, see \cite{humphries_equidistribution_2018}. This conjectural behavior, called the Gaussian Moment conjecture, encapsulates the Quantum Unique Ergodicity conjecture.

This article considers the joint value distribution of Hecke-Maa{\ss} cusp forms on the hyperbolic 3-space $\mathbb{H}^3$, analogously to the recent results of Hua, Huang and Li \cite{huang_joint_2024, hua_joint_2024} for the hyperbolic plane $\mathbb{H}^2$. In the case of the hyperbolic $3$-space, we have the following conjecture (analogue of \cite[Conjecture 1.6]{huang_joint_2024}): 

\begin{conj}
Let $\Gamma$ be an arithmetic subgroup of $\mathrm{PSL}_2(\C)$.
Let $\Omega$ be a fixed compact set of $\Gamma \backslash \mathbb{H}^3$ such that its boundary $\partial \Omega$ is of measure zero. Let $f$ and $g$ be two Hecke-Maa{\ss} forms such that $\langle f, g \rangle = 0$. Then, for any positive integer $a$, we have 
\begin{equation}
\int_{\Omega} f^a \bar{g}^a d\mu(z) = o(1), 
\end{equation}
as the corresponding eigenvalues tend to infinity.
\end{conj}

Even though this conjecture is out of reach by current methods, it can be proven for low moments and we may therefore be interested in the limiting behavior of
\begin{equation}
\int_\Omega f^2 g \qquad \text{or} \qquad \int_{\mathbb{H}^3} \psi f^2 g, 
\end{equation}
where $\psi$ is a smooth compactly supported function.
The following theorem is the main result of this paper and a smoothed analogue of the joint value equidistribution conjecture; it is analogous to \cite[Theorem 1.4]{hua_joint_2024} in the case of the Poincaré upper-half plane $\mathbb{H}^2$.
\begin{thm}
\label{thm}
Let $f$ and $g$ be two Hecke-Maa{\ss} cusp forms. Assume the Generalized Lindelöf Hypothesis. We have, for all $\psi \in C^\infty_c(\Gamma \backslash \mathbb{H}^3)$, as $t_f \to \infty$, 
\begin{equation}
\langle \psi, f^2g\rangle = \langle \psi, g \rangle \mathbf{1}_{t_f > t_g - t_g^\varepsilon}  + O_{\psi}(t_f t_g^{1/2})^{-1+\varepsilon}.
\end{equation}
\end{thm}

\begin{rem} 
A careful study of the arguments in the proof of \Cref{thm} shows that it is sufficient to assume, instead of the Generalized Lindelöf Hypothesis, a strong subconvex bound of the form $L(s,f) \ll c(f)^{1/8-\delta}$ for a $\delta > 0$, for the  $L$-functions occurring (which can be of degree up to $16$). We refer to further remarks at the end of the proof concerning more general number fields and the level aspect analogue.
\end{rem}

\section{Statement and tools}

We introduce in this section the notations used in this article as well as  the tools and results that will be invoked throughout the paper for easier reference. Part of this section might be skipped in a first reading and referred to when necessary.

\subsection{Hyperbolic 3-space and automorphic forms}

We recall the theory of the hyperbolic $ 3 $-space, for which a good reference is \cite{emg}. 
Let 
\begin{equation}
\mathbbm{H}^3 := \{P = z + rj:\ z \in \C, \ j > 0\}
\end{equation}
be the 
hyperbolic $ 3 $-space which we consider as a subset of Hamilton's 
quaternions with the standard $ \R $-basis $(1, i, j, k)$. The $3$-space $ \mathbb{H}^3 $ is equipped with the hyperbolic metric. 
The corresponding volume element is given by 
\begin{align} \label{hypmeasure}
dv = dv(P) = \frac{dx dy dr}{r^3} 
\end{align}
and the corresponding Laplace-Beltrami operator by 
\begin{align}\label{hypoperator}
\Delta = 
r^2 \left(\frac{\partial^2}{\partial x^2} + \frac{\partial^2}{\partial y^2} 
+ \frac{\partial^2}{\partial r^2} \right) - r \frac{\partial}{\partial r}.
\end{align}
The group PSL$_2(\C) $ acts on $ \mathbb{H}^3 $ as follows: if 
$ 
M 
= 
\left(\begin{smallmatrix}a&b\\c&d\end{smallmatrix}\right) 
\in 
\mathrm{PSL}_2(\mathbb{C}) $ and $ P = z + rj \in \mathbbm{H}^3 $ is 
considered as a quaternion whose fourth components equals zero, 
then 
\begin{align*}
M P :=(aP+b)(cP+d)^{-1}, 
\end{align*}
where the inverse is taken in the skew field of quaternions. 
Let $ K = \mathbb{Q}(\sqrt{D}) $, $ D < 0 $, be an imaginary quadratic 
number field of discriminant $ d_K $ and of class number $ H_K = 1 $. 
There are only nine such imaginary quadratic fields. We denote the ring 
of integers of $ K $ by $ \mathcal{O}_K $ and its unit group by 
$ \mathcal{O}_K^{*} $. It is given by
\begin{equation}
\mathcal{O}_K^* = 
\begin{cases}
\{\pm 1\} & \text{if } D \notin \{-1, -3\},  \\ 
\{\pm 1, \ \pm i\} & \text{if } D = -1,  \\ 
\{\pm 1, \pm \rho , \pm \rho^2\}  & \text{if } D = - 3.
\end{cases}
\end{equation}

\begin{rem} 
Working with class number $H_K=1$ simplifies the calculations as we only have to consider one cusp in the spectral decomposition. The methods used to prove Theorem \ref{thm} can nevertheless be applied to the analogous problem for other imaginary quadratic fields, up to dealing with the existence of other cusps in the case of class number greater than one.
\end{rem}

Let $ \Gamma := \Gamma_K := \textup{PSL}_2(\mathcal{O}_K) $. This is 
a co-finite subgroup of PSL$_2(\C) $. As $ H_K = 1 $, $ \Gamma $ has only 
one cusp. The spectral theory of $ -\Delta $ on $ \mathcal{M}_{\Gamma} = 
\Gamma \setminus \mathbb{H}^3 $ is well-known (see e.g.\ \cite{emg}). 
As $ \Gamma $ is not co-compact but co-finite, the spectrum of $ - \Delta $
consists of a discrete part of finite multiplicity and an absolutely 
continuous part of multiplicity~$1$. The 
absolutely continuous part can be described using Eisenstein series 
which are defined as follows: let $ \Gamma_{\infty} = \{\gamma \in 
\Gamma:\ \gamma \infty = \infty\} $ be the stabilizer of the cusp 
$ \infty $ and $ 
\Gamma'_{\infty} = 
\{\gamma \in \Gamma_{\infty}: |\textup{tr} \, \gamma| = 2\} 
$ 
to be the set of all parabolic elements of $ \Gamma $ that stabilize 
$ \infty $; in particular $ \Gamma'_{\infty} $ is a normal subgroup of 
$ \Gamma_{\infty} $. Note that $ \Gamma_{\infty} \not= \Gamma'_{\infty} $ 
only if $ D = -1 $ and $ D = -3 $. In these cases $ \Gamma_{\infty} $ 
also contains elliptic elements of $ \Gamma $. The Eisenstein series 
is defined by 
\begin{equation} 
\label{eq:eisenstein-series-def}
E(P, s) 
= \sum_{\gamma \in  \Gamma'_{\infty} \setminus \Gamma} r(\gamma P)^{1+s}, 
\quad  \Re(s) > 2 .
\end{equation} 
The theory of Eisenstein series for the hyperbolic $ 3 $-space can be found 
in \cite{emg} or \cite{hei}. A different way to define the Eisenstein series is by
\begin{equation*} 
E_{\infty}(P, s) 
= \sum_{\gamma \in  \Gamma_{\infty} \setminus \Gamma} r(\gamma P)^{1+s}, 
\quad \Re(s) > 2, 
\end{equation*} 
and then $ E(P, s) = \frac{|\mathcal{O}_K^*|}{2} E_{\infty}(P, s) $, see 
\cite[p. 232]{emg}. In order to write the Fourier expansion of the 
Eisenstein series let $ \zeta_K(s) $ denote the Dedekind 
zeta function for $ K $ and 
\begin{equation} \label{eq:scattering_H3} 
\phi(s) := \frac{2 \pi}{s \sqrt{|d_K|}} \frac{\zeta_K(s)}{\zeta_K(1+s)} 
\end{equation} 
be the scattering matrix for $ \textup{PSL}_2(\mathcal{O}_K) $. Then 
the Fourier expansion of $ E_{\infty}(P, s) $ is given by 
\begin{equation} \label{eq:FE_Eisenstein_H3} 
\begin{split} 
E_{\infty}(P, s) &= 
r^{1+s} + \phi(s) r^{1-s} 
+ \frac{2 (2 \pi)^{1+s}}{|d_K|^{(1+s)/2} \Gamma(1+s) \zeta_K(1+s)} \\ 
& \qquad \times
\sum_{0 \not= \omega \in \mathcal{O}_K} |\omega|^s \sigma_{-s}(\omega) 
r K_s\left(\frac{4 \pi |\omega| r}{\sqrt{|d_K|}}\right) e^{2 \pi i 
\left\langle \frac{2 \overline{\omega}}{\sqrt{d_K}}, z\right\rangle} 
\end{split} 
\end{equation} 
as proven in e.g. \cite[p. 102]{hei}. Here $ \sigma_s(\omega) $ denotes the generalized divisor 
function 
\begin{equation*} 
\sigma_s(\omega) = \frac{1}{|\mathcal{O}_K^{*}|} 
\sum_{\substack{d \in \mathcal{O}, \\ d \mid \omega}} 
|d|^{2s} 
\end{equation*} 
and $ \langle \mu, z\rangle = \Re (\overline{\mu} z) $. 
In the three-dimensional case the Hecke operators are 
defined as follows: if $ n \in \mathcal{O}_K \setminus \{0\} $ 
we define $ \mathcal{M}_n $ to be the set of all matrices of 
the form 
$ 
\left(\begin{smallmatrix} a & b \\ c & d \end{smallmatrix}\right) 
$, 
$ ad - bc = n $. Then for $ f $ being a $ \Gamma $-invariant 
function the Hecke operator $ T_n $ is given by 
\begin{equation*} 
(T_n f)(P) := 
\frac{1}{\sqrt{N(n)}} 
\sum_{\gamma \in \Gamma \setminus \mathcal{M}_n} f(\gamma P). 
\end{equation*} 
Here $ N(n) = |n|^2 $ is the norm of $ n \in \mathcal{O}_K $. 
Furthermore, the action of
$
\mbox{GL}_2(\C) 
:= 
\{M \in \mbox{$\mbox{M}_2(\C)$}: \, \mbox{det} \, 
M \not= 0 \} $ on $ \mathbbm{H}^3 $ is given as follows. If 
$ 
M = 
\left(\begin{smallmatrix}
a & b \\
c & d
\end{smallmatrix}\right) 
\in \mbox{\textup{GL}}_2(\C) $ and $ P \in \mathbbm{H} $, 
we set $ q := \sqrt{\mbox{\textup{det}} \, M } $ and 
$$ 
MP := q^{-1} (aP + b)(cP + d)^{-1} q,
$$
with the inverses being taken in the skew field of quaternions. 
The theory for Hecke operators for Bianchi groups is developed in 
\cite{hei}. 
However, in contrast to Heitkamp we have incorporated the factor 
$ 1/\sqrt{N (n)} $ in the definition of the Hecke operator. 

The Fourier expansion of a cusp form is given in \cite[Lemma 16.1]{hei}. Let $f$ be an automorphic form with eigenvalue $\lambda = 1-s^2$. Then we have the Fourier expansion at the cusp $\infty$ given by
\begin{equation}
\label{eq:fourier-expansion-cusp}
f(z+rj) = \phi(r,s) + \sum_{0 \neq \mu  \in \mathfrak{o}^2} \rho_f(\mu) r K_s(2\pi |\tilde{\mu}|r) e(\langle \tilde{\mu}, z\rangle), 
\end{equation}
where 
$\tilde{\mu} = (2/\sqrt{d_K})\bar{\mu}$. If the form is $L^2$-normalized, as it will be in this article, then we can deduce the value of its first coefficient $\rho_f(1)$ by the Rankin-Selberg method: 
\begin{equation}
\label{eq:coeff-1}
|\rho_f(1)|^{-2} = \frac{|d_K|^{1/2}}{4}\Lambda(1, \mathrm{sym}^2 f), 
\end{equation}
which should be compared to \cite[Korollar 17.6]{hei}, noting that the convention therein is the arithmetic normalization $\rho_f(1)=1$ and therefore a non-$L^2$-normalized form. We emphasize that the numerator $32\pi^4$ in \cite[p. 125, last equation]{hei} should instead be~$128\pi^4$.

\subsection{Preliminary results}

A central tool  in our method is the spectral expansion and the Parseval equality, which reads as follows \cite[Section 6.3, Theorem 3.4]{emg}.
\begin{lem}[Spectral expansion]
\label{lem:spectral-expansion}
For any $\psi \in L^2(\Gamma \backslash \mathbb{H}^3)$, we have
\begin{equation}
\psi = \frac{\langle \psi, 1\rangle}{\mathrm{vol}(\Gamma \backslash \mathbb{H}^3)} + \sum_m \langle \psi, u_m \rangle u_m + \frac{1}{\pi \sqrt{|d_K|} |\mathcal{O}_K^\times|} \int_{\mathbb{R}} \langle \psi, E(\cdot, it)\rangle E(\cdot, it)dt, 
\end{equation}
where the sum is over a Hecke orthonormal basis $(u_m)_m$ of cusp forms. 
\end{lem}

We will need to explain the projections occurring when applying this spectral expansion to $f^2g$ as in the statement of Theorem \ref{thm}, and these involve triple inner products of cusp forms. To control these terms, we appeal to the Watson-Ichino formula and use more precisely the following uniform bound.

\begin{lem}[Watson-Ichino formula]
\label{lem:watson}
There is an absolute constant $C$ such that for any $\phi_1, \phi_2, \phi_3$ three spherical Hecke-Maa{\ss} forms, which we consider as spherical vectors in the associated automorphic representations $\pi_1, \pi_2, \pi_3$, we have 
\begin{equation}
\label{eq:watson}
|\langle \phi_1, \phi_2 \phi_3 \rangle|^2  := \left| \int_{\Gamma \backslash G } \phi_1 \phi_2 \phi_3 \right|^2 = \frac{C}{8\pi} \frac{\Lambda(\tfrac12, \pi_1 \otimes \pi_2 \otimes \pi_3)}{\prod_i \Lambda(1, \mathrm{sym}^2 \pi_i)}.
\end{equation}
\end{lem}

\begin{proof}
This is in line with \cite[(6.5)]{nicole}, with the precise computation of the ``constants" appearing therein. See also the very recent paper \cite{nelson_subconvexity_2025} explaining bounds on triple products in the real case of the quantum unique ergodicity conjecture. 

For each $i \in \{1, 2, 3\}$, let $\pi_i$ be a cuspidal spherical representation in which $\phi_i \in \pi_i$  is a spherical vector. By the Ichino formula \cite[Theorem 4]{marshall_triple_2014}, there is an absolute constant $C$ such that 
\begin{equation}
\label{wi}
\left| \int_{\Gamma \backslash G } \phi_1 \phi_2 \phi_3 \right|^2 = C \int_G \prod_{i=1}^3 \langle \pi_i(g) \phi_i, \phi_i \rangle dg \frac{L(\tfrac12, \pi_1 \otimes \pi_2 \otimes \pi_3)}{\prod_i L(1, \mathrm{sym}^2 \pi_i)}.
\end{equation}
The result therefore reduces to proving that the above integral over $G$ of product of matrix coefficients is, up to a constant, equal to the Gamma factors occurring in the completed $L$-functions \eqref{eq:watson}. 

We embed $\pi_1$ and $\pi_2$ in their respective Whittaker models: denote by $\mathcal{W}(\pi_i, \psi)$ the Whittaker model of $\pi$ with respect to the character $\psi$ \cite[Section II.2.8]{bump}. Let $W_i \in \mathcal{W}(\pi_i, \psi)$ be a Whittaker function corresponding to $\phi_i$.
We embed $\pi_3$ in its induced model $\mathcal{I}_3$ as explained in \cite[Theorem II.2.7.1]{bump}, and denote $f_3 \in \mathcal{I}_3$ the vector corresponding to $\phi_3$, which is given by  \cite[(5.13) and (5.22)]{bump}, viz.
$$
f_3(n(x)a(y)k) = x^{s_1+s_2} y^s \phi_3(k) = y^{1+it} f_3(k)
$$
for all $g = n(x)a(y)k$ decomposed according to the Iwasawa decomposition of $G$, where $n(x)$ is the unipotent matrix with upper-right entry $x$ and $a(y)$ is the diagonal matrix parametrized by $y$. Here, $\lambda = \sigma(1-\sigma)$ relates the spectral parameter $\sigma$ to the eigenvalue $\lambda$, where $\sigma = \tfrac12(s_1 - s_2 + 1) = 1+it$. 
By using \cite[Lemma 3.4.2]{michel_subconvexity_2010} also stated in \cite[Proposition 7]{marshall_triple_2014}, we can express in these terms the sought integral: 
$$
\int_G \prod_{i=1}^3 \langle \pi_i(g) \phi_i, \phi_i \rangle dg = \frac{|\mathcal{T}|^2}{8\pi} 
$$
where
\begin{equation}
\label{eq:T-def}
\mathcal{T} := \int_0^\infty \int_K W_1(a(y)k) \overline{W_2}(a(y)k) y^{1+it} f_3(k) y^{-2} dk d^\times y.
\end{equation}

We now follow the strategy used in \cite[Proposition 6]{marshall_triple_2014} to deal with the integral over $K$. We examine the action of the central torus $M=K\cap Z$ on $W_i$ and $f_3$, for which these are eigenfunctions. Concerning $W_i$, the left action under $K$ is given in \cite[Proposition II.2.8.1]{bump} (it has a weight $k_i$ in the notations therein), and the right action under $Z$ is given in \cite[p. 244]{bump} for diagonal matrices (it has weight $\mu_i$); therefore the action under $M = K\cap Z$ has necessarily weight $k_i=\mu_i$. As for $f_3$, it has a given weight on both sides by the explicit computations from \cite[(5.22)]{bump}. 

The vector $W_i$ (resp. $f_3$) has a certain weight $w_i$ under the right action of $M$,  therefore we require $w_1 - w_2 + w_3 = 0$ for the expression \eqref{eq:T-def} to be nonzero, otherwise a change of variable $k \mapsto km$ does show that the integral over $K$ vanishes. This means that  the function $W_1(a(y)k) \overline{W_2}(a(y)k) f_3(k)$ is $K$-invariant and therefore the value of the integrand is constant, thus the integral is equal to the value of this function at e.g. $k=\mathrm{id}$. We therefore obtain
\begin{equation}
\mathcal{T} = \int_0^\infty W_1(a(y)) \overline{W_2}(a(y)) y^{-1+it} d^\times y.
\end{equation}

We appeal to the explicit knowledge of the Whittaker functions in our case \cite[Proposition 8, with weights all zero and therefore $m=0$]{marshall_triple_2014}, which is also given in \cite[(22), in the case $k=0$]{marshall_triple_2014}. We have 
\begin{equation}
W_i(a(y)) = |\Gamma(1+it_i)|^{-1} y K_{it_i}(4\pi y).
\end{equation}

We input these values in $\mathcal{T}$ and use the inversion formula  given in \cite[6.576]{gr} to obtain
\begin{equation}
\label{eq:K-formula}
\int_0^\infty y^\lambda K_\mu(y) K_\nu(y) dy = \frac{1}{\Gamma(\lambda+1)} \prod_{\pm, \pm} \Gamma\left( \frac{1 + \lambda \pm \mu \pm \nu}{2}\right), 
\end{equation}
where the product is over both $\pm$ and therefore has four factors. All in all, we get
\begin{equation*}
\mathcal{T} = \frac{\Gamma(\tfrac{1 + it + it_1 + it_2}{2})\Gamma(\tfrac{1+ it + it_1 - it_2}{2})\Gamma(\tfrac{1 + it - it_1 + it_2}{2})\Gamma(\tfrac{1+ it - it_1 - it_2}{2})}{|\Gamma(1+it_1)\Gamma(1+it_2)|\Gamma(1+it)}.
\end{equation*}
Therefore, $|\mathcal{T}| = \mathcal{T}\overline{\mathcal{T}}$ exactly features the eight Gamma factors needed to complete the $L$-functions arising in \eqref{wi}, finishing the proof.
\end{proof}


Triple products involving Eisenstein series will also occur when spectrally expanding $f^2g$ in Theorem \ref{thm}, and these can be controlled by appealing to the unfolding technique.

\begin{lem}[Rankin-Selberg unfolding]
\label{lem:rankin-selberg}
For every $t \in \mathbb{R}$, denote $E_t := E(\cdot, it)$. For every $t, s \in \mathbb{R}$ and every cusp forms $u_k$ and $g$, we have
\begin{align}
\left\langle u_kg, E_t \right\rangle &= \frac{\rho_k(1)\rho_g(1) \Lambda(\frac{1+it}{2}, g \times u_k)}{2\Lambda(1+it)},  \\ 
\left\langle E_tg, E_s \right\rangle &= \frac{\rho_t(1)\rho_g(1) \Lambda(\frac{1+is+it}{2}, g)\Lambda(\tfrac{1+is-it}{2}, g)}{2\Lambda(1+it)}.
\end{align}
\end{lem}

\begin{rem}
This is the analogue of the Watson-Ichino formula when  some cuspidal Hecke-Maa{\ss} forms are replaced by Eisenstein series. Taking the squared modulus of the above inner product and using \eqref{eq:coeff-1}, we indeed obtain 
\begin{align}
|\left\langle u_kg, E_t \right\rangle|^2 & =  \frac{ 4 \Lambda(\tfrac12, E_t \times g \times u_k)}{|d_K||\Lambda(1+it)|^2\Lambda( 1, \mathrm{sym}^2 g)\Lambda(1, \mathrm{sym}^2 u_k)}, \\ 
|\left\langle E_tg, E_\tau \right\rangle|^2 & =  \frac{ 4\Lambda(\tfrac12, E_t \times E_\tau \times g)}{|d_K||\Lambda(1+it) \Lambda (1+i\tau)|^2 \Lambda(1, \mathrm{sym}^2 g)}.
\end{align} 
\end{rem}

\begin{proof}
We start by $\langle u_kg,  E_t \rangle$, which we will unfold by inputting the Fourier expansion of each cusp form and Eisenstein series. This generalizes \cite[Section 17.2]{hei}, where the computation is restricted to the case $u_k = g$.  By using the Fourier expansions \eqref{eq:fourier-expansion-cusp} of $u_k$ and $g$, and by denoting $t_k$ and $t_g$ their respective spectral parameters, we have 
\begin{align}
\langle u_k, g E_t \rangle & = \int_F \left( \sum_{0 \neq \mu \in \mathfrak{o}^2} \rho_k(\mu) r K_{it_k}(2\pi |\tilde{\mu}| r)e(\langle \tilde{\mu}, r \rangle)\right) \\
&  \quad \times\left( \sum_{0 \neq \nu \in \mathfrak{o}^2} \rho_g(\nu) r K_{it_g}(2\pi |\tilde{\nu}| r)e(\langle \tilde{\nu}, r \rangle)\right) E_t(z+rj) dz r^{-3} dr.
\end{align}
By replacing the Eisenstein series $E_t$ by its definition \eqref{eq:eisenstein-series-def} and unfolding the integral, we get
\begin{align}
\langle u_k, g E_t \rangle & = \int_0^\infty \int_{P_{\mathfrak{o}}} \left( \sum_{0 \neq \mu \in \mathfrak{o}^2} \rho_k(\mu)  K_{it_k}(2\pi |\tilde{\mu}| r)e(\langle \tilde{\mu}, r \rangle)\right) \\
&  \quad \times\left( \sum_{0 \neq \nu \in \mathfrak{o}^2} \overline{\rho_g(\nu) K_{it_g}(2\pi |\tilde{\nu}| r)e(\langle \tilde{\nu}, r \rangle)}	\right)  dz r^{it} dr.
\end{align}
By executing the integration of $e(\langle \tilde{\mu}-\tilde{\nu}, z \rangle)$ over $z$ first, we select only the case where $\mu = \nu$ with a weight $\mathrm{vol}(P_{\mathfrak{o}})$; indeed by \cite[p. 122]{hei} we have
\begin{equation}
\label{orthogonality}
\int_{P_{\mathfrak{o}}} e(\langle \tilde{\mu} - \tilde{\nu}, z \rangle) dx dy = \mathrm{vol}(P_{\mathfrak{o}}) \mathbf{1}_{\nu = \mu}.
\end{equation}
We therefore obtain
\begin{align}
\langle u_k, g E_t \rangle & =  \mathrm{vol}(P_{\mathfrak{o}}) \sum_{0 \neq \mu \in \mathfrak{o^2}} \rho_k(\mu) \overline{\rho_g(\mu)} \int_0^\infty K_{it_k}(2\pi |\tilde{\mu}|r) K_{it_g}(2\pi |\tilde{\mu}| r) r^{it} dt dr \\
& = \left( \frac{2}{\sqrt{d_K}}\right)^t \frac{\mathrm{vol}(P_{\mathfrak{o}})}{4\pi^{t} \Gamma(t+1)}   \prod_{\pm} \Gamma\left( \frac{1+it \pm it_k \pm it_g}{2} \right)  \sum_{0 \neq \mu \in \mathfrak{o^2}} \frac{\rho_k(\mu) \overline{\rho_g(\mu)}}{|\mu|^{1+it}}
\end{align}
where the Mellin transform of product of Bessel functions has been replaced by Gamma factors using \eqref{eq:K-formula}, which is valid under the condition $\Re(t) > |\Re(it_k)| + |\Re(it_g)| = 0$. Note moreover that $\mathrm{vol}(P_{\mathfrak{o}}) = \tfrac12 \sqrt{|d_K|}$. We conclude by using the fact that
\begin{equation}
\sum_{0 \neq \mu \in \mathfrak{o^2}} \frac{\rho_k(\mu) \overline{\rho_g(\mu)}}{|\mu|^{t}} = \rho_k(1) \rho_g(1) \frac{L(t, f \times g)}{\zeta_k(2t)}, 
\end{equation}
and by the defintion of the completing factors: 
\begin{align*}
\Lambda(s) & = (2\pi/\sqrt{d_K})^{-s} \Gamma(s) \zeta_K(s), \\
\Lambda(t, f \times g) & = (2\pi/\sqrt{d_K})^{-4t} \prod_{\pm, \pm}\Gamma(t\pm\tfrac{it_f}{2}\pm\tfrac{it_g}{2}) L(t, f \times g).
\end{align*}

We proceed similarly in the case of two Eisenstein series. By inputting the Fourier expansions of the cusp form \eqref{eq:fourier-expansion-cusp} and one Eisenstein series \eqref{eq:FE_Eisenstein_H3}, while using the definition of the second Eisenstein series to unfold the integration domain, we get
\begin{align*}
\langle fE_t, E_s \rangle & = \int_0^\infty \int_{P_{\mathfrak{o}}} \left( r \sum_{0 \neq \mu} \rho_f(\mu) K_{it}(2\pi |\tilde{\mu}|r) e(\langle \tilde{\mu}, z \rangle) \right) \\
& \quad \times \left( r^{is} + \phi(is) r^{2-is} + \frac{2r}{\xi_K(s)}\sum_{0 \neq m} |m|^{is-1} \sigma_{1-is}(m) K_{is-1}(2\pi |\tilde{m}| r) e(\langle \tilde{m}, z\rangle) \right) r^{it-3} dzdr
\end{align*}

By executing the integration of $e(\langle \tilde{\mu}-\tilde{m}, z \rangle)$ over $z$ first and using \eqref{orthogonality}, we  only pick up the terms where $\mu = m$ with a weight $\mathrm{vol}(P_{\mathfrak{o}})$, and this in particular removes the constant terms of the Eisenstein series' expansion. We are therefore left with
\begin{align*}
\langle fE_t, E_s \rangle & = \frac{2\mathrm{vol}(P_{\mathfrak{o}})}{\xi_K(is)}  \sum_{0 \neq \mu} \rho_f(\mu) |\mu|^{is-1} \sigma_{1-is}(\mu)  \int_0^\infty  K_{it}(2\pi |\tilde{\mu}|r)   K_{is-1}(2\pi |\tilde{\mu}| r)  r^{it-1} dr.
\end{align*}

Appealing again to the explicit Mellin transform of products of Bessel functions in terms of Gamma factors, we obtain
\begin{align}
\langle E_t f, E_s \rangle & = \rho_f(1)2^{2-(1-it)}  \frac{\sqrt{|d_K|}}{2} \frac{|\mathcal{O}^\times|^2}{\xi_K(it)\xi_K(is)}  \left( \frac{4\pi}{\sqrt{|d_K|}}\right)^{-it} \\
& \quad \times  \frac{L(\frac{it-is+1}{2}, f)L(\frac{it+is+1}{2}, f)}{\Gamma(it)} \prod_{\pm, \pm}  \Gamma\left(\frac{it\pm (is-1)\pm it_f -1}{2}\right) \notag
\end{align}
where we can read off the L-factors as well as the completing gamma factors.
\end{proof}

We recall the Stirling formula that will be used to estimate the Archimedean parts of the L-functions arising from Watson-Ichino formula and the Rankin-Selberg method.
\begin{lem}[Stirling formula]
\label{lem:stirling}
For bounded $x \in \R$, as $y \to \infty$, we have 
\begin{equation}
|\Gamma(x+iy)| \asymp (1+|y|)^{x-\frac{1}{2}} e^{-\frac{\pi}{2}|y|}.
\end{equation}
\end{lem}

We will use uniform bounds on Eisenstein series, and we prove the following.
\begin{lem}[Uniform bounds on Eisenstein series]
\label{lem:eisenstein-bound}
For any compact $\Omega \subset \mathbb{H}^3$ and any $\epsilon>0$, we have
\begin{equation}
\label{eq:Eisenstein-sup-bound}
    \sup_{z \in \Omega} \;| E(z,it)|\ll_{\Omega, \epsilon} (1+|t|)^{1+\epsilon} .
\end{equation}
\end{lem}

\begin{rem}
 Assing \cite{assing_sup-norm_2019} recently proved a stronger bound on sup-norm for Eisenstein series with exponent $3/8$, matching the quality of \cite{young_note_2018}; however the above lemma is sufficient to our purposes and we provide an elementary proof for completeness.
 \end{rem}

\begin{proof}
    We follow closely \cite{young_note_2018}. Consider the Fourier expansion given in \eqref{eq:FE_Eisenstein_H3} and the relation $E=\frac{|\mathcal{O}_K^*|}{2}E_{\infty}$. After applying the Cauchy-Schwarz inequality, it suffices to bound \[C(t)=\frac{1}{|\Gamma(1+it)||\zeta_K(1+it)|} \cdot\sum_{0 \not= \omega \in \mathcal{O}_K} \sigma_{0}(\omega) 
     \left|K_{it}\left(\frac{4 \pi |\omega| r}{\sqrt{|d_K|}}\right)\right|,  \] the others terms being bounded with respect to $\Omega$. By the reflection formula for the Gamma function we can rewrite it as \[C(t) \ll \frac{1}{|\zeta_K(1+it)|} \cdot \sum_{0 \not= \omega \in \mathcal{O}_K} \sigma_{0}(\omega) 
     \left|K_{it}\left(\frac{4 \pi |\omega| r}{\sqrt{|d_K|}}\right)\right| \frac{\cosh(\pi t/2)}{|1+t|^{1/2}}.\]
     To simplify the notation, denote $L=4\pi r/\sqrt{|d_K|}$ and notice that both $L$ and $1/L$ are bounded in terms of $\Omega$. We need to state uniform asymptotic estimates for the $K$-Bessel function, that we encapsulate in the following lemma. 

\begin{lem}
Let $C \gg 1$. We have 
\begin{align}
\label{eq:KBesselBalogh}
 \cosh(\pi t/2) K_{it}(u) \ll
 \begin{cases}
  t^{-1/4} (t -u)^{-1/4}, \quad  &\text{if }   0 < u < t - C t^{1/3},  \\
  t^{-1/3}, \quad &\text{if } |u-t| \leq C t^{1/3},
 \\ 
u^{-1/4} (u -t)^{-1/4}  \quad  &\text{if }    \tfrac{\pi}{2}t \geq u > t + C t^{1/3}, \\
u^{-1/4} (u -t)^{-1/4} \exp(-u+\pi t/2) \quad & \text{if } u> \tfrac{\pi}{2}t.
 \end{cases}
\end{align}
\end{lem}     

\begin{proof}
To obtain the second estimate, we borrow the expansions from Balogh \cite[Equation (8)]{balogh_asymptotic_1967} in the transition regime $|u-t| \leq C t^{1/3}$, and note that the Airy function therein is uniformly bounded in a neighborhood of zero.

In the regime $0<u<t-Ct^{1/3}$, we  appeal to Erdélyi \cite[7.13.2 (19) with $M=1$]{erdelyi},  noting that the exponential is asymptotically $\cosh(\pi t/2)$, giving the first claimed estimate.

In the regime $u>\tfrac{\pi}{2} t$, we appeal to Erdélyi \cite[7.13.2 (18) with $M=1$]{erdelyi} in order to obtain extra decay.  Asymptotically expanding the argument in the exponential gives that the exponential is $\ll \cosh(\pi t / 2) \exp(-u + \pi t /2)$ as claimed in the fourth stated estimate.

In the regime $\tfrac{\pi}{2} t > u>t+Ct^{1/3}$, we appeal to Erdélyi \cite[7.13.2 (18) with $M=1$]{erdelyi} . Let $x=u/t$. In the range $x \geq 1$, the argument of the exponential is $-t\left(\sqrt{x^2-1}+\arcsin(1/x)\right)$. We apply the identity $\pi/2-\arcsin(1/x)=\arctan\left(\sqrt{x^2-1} \right)$, valid for $x\geq 1$, to rewrite the argument as $-\pi t/2 - t( v - \arctan(v)) $, for $v=\sqrt{x^2 -1}$. We recall that the contribution of the argument $-\pi t/2$ exactly corresponds to the one of $1/\cosh(\pi t / 2)$. Finally, we notice that $v- \arctan{v} \geq 0$, for $v \geq 0$, since the derivative of the left hand side is non-negative and the inequality holds for $v=0$.
\end{proof}
We split the sum over $\omega_g$  by considering the following three cases.  The contribution of the range $L|\omega|\leq \frac{t}{2}$ is bounded by
\begin{equation}
\label{eq:FourierExpansionBulkBound}
 \frac{1}{|\zeta_K(1+2it)|} \sum_{|\omega| \leq t/2L} \frac{\sigma_0(\omega)}{t} \ll_{\Omega} t \log^2 t,
\end{equation}
using $|\zeta_K(1+2it)|^{-1} \ll \log t$. 
The contribution of the range $|L|\omega| - t| \leq t^{1/3}$  gives
\begin{equation}
 \frac{1}{|\zeta_K(1+2it)|}  t^{-1/3} \sum_{|L|\omega| - t| \leq t^{1/3}} \frac{\sigma_0(\omega)}{t^{1/2}} \ll_{\Omega} 
 \frac{\log t}{t^{5/6}}   
 \Big(t^{4/3}\log t \Big)\ll_{\Omega} t^{1/2}\log^2t 
\end{equation}
using a trivial bound on the number of $\omega \in \mathcal{O}_K$ in the range. The case $|L|\omega| - t |\asymp \Delta$ for $t^{1/3}\ll \Delta \ll t$ gives
\begin{equation}
 \frac{1}{|\zeta_K(1+2it)|}  (t\Delta)^{-1/4}t^{1/2} \sum_{|\omega| = \frac{t}{L} + O(\frac{\Delta}{ L})} \frac{\sigma_0(\omega)}{t} \ll_{\Omega} 
 \frac{\log t}{(t\Delta)^{1/4}}   
 \Big(t^{1/2}\frac{\Delta}{\log t} + t^{\varepsilon} \Big),
\end{equation}
using Shiu's bound \cite{shiu_brun-titschmarsh_1980} on the function $\tilde{\sigma}_0(n)=\frac{1}{\sqrt{n}}\sum_{|\omega|^2 = n}\sigma_0(\omega)$.
These terms then give the same bound as \eqref{eq:FourierExpansionBulkBound}, plus an additional term that is $\ll_{\Omega} t^{-1/3 + \varepsilon}$.
The contribution from the remaining range with $u > 2t$ is easily bounded thanks to the exponential contribution in \eqref{eq:KBesselBalogh}. 
\end{proof}

\section{Proof of the theorem}

We have now all the tools necessary to prove Theorem \ref{thm}. We follows the strategy of Hua, Huang and Li \cite{hua_joint_2024}.

\subsection{Spectral expansion and first truncation}

We want to determine more precisely  the asymptotic behavior of
\begin{equation}
\langle \psi, f^2 g\rangle := \int_{\Gamma \backslash \mathbb{H}^3} \psi f^2 g,
\end{equation}
where $f$ and $g$ are Hecke-Maa{\ss} forms, and where $\psi \in C^\infty_c(\Gamma \backslash \mathbb{H}^3)$.
Apply Parseval formula from \Cref{lem:spectral-expansion}  to expand $\psi$ spectrally, so that the above integral can be rewritten as
\begin{equation}
\label{eq:spectral-expansion-1}
\frac{\langle \psi, 1\rangle \langle 1, f^2g\rangle}{\mathrm{vol}(\Gamma \backslash \mathbb{H}^3)} + \sum_m \langle \psi, u_m \rangle \langle u_m, f^2g \rangle + \frac{1}{\pi \sqrt{|d_K|}|\mathcal{O_K^\times}|}\int_{\mathbb{R}} \langle \psi, E(\cdot,it) \rangle \langle E(\cdot, it), f^2g\rangle dt.
\end{equation}

The first inner product $\langle \psi, 1\rangle$ is the total mass of $\psi$, i.e. the volume on which we are concentrating. 
Using the compact support and bounds on $\psi$ as well as the symmetry of $\Delta$, we get that the coefficients $\langle \psi, u \rangle$ (no matter whether $u$ is a Maa{\ss} form or an Eisenstein series) decay faster than $t_u^{-A}$ for all $A>0$. 
Indeed, using the fact that $\Delta u_k = \lambda_k u_k$, we obtain
\begin{align}
\langle \psi, u_k\rangle & = \lambda_k^{-\ell} \langle \psi, \Delta^\ell u_k \rangle = \lambda_k^{-\ell} \langle \Delta^\ell \psi, u_k \rangle \ll \lambda_k^{-\ell} \int |u_k| \cdot |\Delta^\ell \psi|
\end{align}
and, using the Cauchy-Schwarz inequality,  the fact that $\|\Delta^\ell \psi \|_2 \ll 1$ (since $\psi$ is $C^\infty_c$ and fixed) and $\|u_k\|_2 = 1$ (since we chose the forms to be $L^2$-normalized), we obtain that $\langle \psi, u_k\rangle \ll \lambda_k^{-\ell}$ for any $\ell > 0$. 

Analogously, we have for the Eisenstein series:
\begin{align}
\langle \psi, E_t\rangle & = \lambda_t^{-\ell} \langle \psi, \Delta^\ell E_t \rangle = \lambda_t^{-\ell} \langle \Delta^\ell \psi, E_t \rangle \ll \lambda_t^{-\ell} \int |E_t| \cdot |\Delta^\ell \psi|.
\end{align}
Using the  Cauchy-Schwarz inequality, the fact that $\|\Delta^\ell \psi \|_1 \ll 1$  and $\|E_t\|_\infty \ll t^{1+\varepsilon}$ from Lemma \ref{lem:eisenstein-bound}, we obtain  $\langle \psi, E_t \rangle \ll \lambda_t^{-\ell}$ for any $\ell > 0$. 

Moreover, by applying the Parseval formula to the other terms arising in \eqref{eq:spectral-expansion-1} which are a product of four terms, we obtain $\langle u_k, f^2g\rangle \ll (t_kt_ft_g)^B$ for a certain $B>0$, and $\langle E_t, f^2 g\rangle \ll (t_f t_g t)^B$ for a certain $B>0$ (see also the Section \ref{sec} for a similar argument). These bounds are easily balanced by the above bounds from the terms involving $\psi$. Therefore, this argument truncates the effective spectral range in \eqref{eq:spectral-expansion-1} to the very small ranges $t_m \ll \max(t_f, t_g)^\varepsilon$ and $t \ll \max(t_f, t_g)^\varepsilon$.

We are therefore left with
\begin{align}
\label{eq:spectral-expansion-2}
\langle \psi, f^2 g\rangle & = \frac{\langle \psi, 1\rangle \langle 1, f^2g\rangle}{\mathrm{vol}(\Gamma \backslash \mathbb{H}^3)} + \delta_{t_g \ll t_f^{o(1)}} \langle \psi, g \rangle \langle 1, f^2 g^2 \rangle + \sum_{\substack{u_k \neq g \\ t_k \ll \max(t_f,t_g)^{o(1)}}} \langle \psi, u_k \rangle \langle u_k, f^2 g\rangle \notag \\
& \quad  + \frac{1}{\pi \sqrt{|d_K|}|\mathcal{O_K^\times}|} \int_{|t|\ll \max(t_f, t_g)^{o(1)}} \langle \psi, E_t\rangle \langle E_t , f^2g \rangle dt, 
\end{align}
where we singled out the term $u_m = g$, for it will give the potential main term in Theorem~\ref{thm}.

\subsection{Reduction to triple products}
\label{sec}

Many terms in \eqref{eq:spectral-expansion-2} display inner products involving \textit{four} forms. In order to make them amenable to the Watson-Ichino formula, we apply further the spectral expansion from Lemma \ref{lem:spectral-expansion} to the terms $u_kg$ and $E_t g$, which allows to reduce the expression to products of \textit{three} forms at the cost of an extra spectral summation. We obtain
\begin{align}
\label{eq:spectral-expansion-3}
& \langle \psi, f^2g \rangle  = \langle \psi, 1 \rangle \langle 1, f^2 g\rangle  \\ \nonumber
& \quad + \delta_{t_g \ll t_f^{o(1)}} \langle \psi, g \rangle \left( \langle f^2, 1 \rangle \langle 1, g^2 \rangle + \sum_j \langle f^2, u_j \rangle \langle u_j, g^2 \rangle + \int \langle g^2, E_t \rangle \langle E_t, f^2 \rangle dt \rangle \right) \\ \nonumber
& + \sum_{\substack{u_k \neq g \\ t_k \ll \max(t_f, t_g)^{o(1)}}} \langle \psi, u_k \rangle \left( \underbrace{\langle gu_k, 1 \rangle}_{=0} + \sum_j \langle gu_k, u_j \rangle \langle u_j, f^2 \rangle + \int \langle u_k g, E_t \rangle \langle E_t, f^2 \rangle dt \right) \\ \nonumber
& + \int_{|t|\ll \max(t_f, t_g)^{o(1)}} \langle \psi, E_t \rangle \left( \underbrace{\langle gE_t, 1 \rangle}_{=0} + \sum_j \langle gE_t, u_j \rangle \langle u_j, f^2 \rangle + \int \langle gE_t, E_s \rangle \langle E_s, f^2 \rangle ds \right) \\ \nonumber
& + O(\max(t_f, t_g)^{-A}).
\end{align}

We need apply the Watson-Ichino formula and further spectral analysis to control the remaining inner products not involving the very well-behaved $\psi$. For the constant term, this is directly the Watson-Ichino formula from \Cref{lem:watson}, which reads
\begin{align}
|\langle 1, f^2g \rangle|^2 = |\langle f^2, g\rangle |^2 = \frac{\Lambda(\tfrac12, \mathrm{sym}^2 f \otimes g) \Lambda(\tfrac12, g)}{\Lambda(1, \mathrm{sym}^2 f)^2 \Lambda(1, \mathrm{sym}^2 g)}.
\end{align}

For the quadruple products $\langle u, f^2 g\rangle = \langle ug, f^2\rangle$, by \eqref{eq:spectral-expansion-3} we need to estimate more precisely quantities of the form $\langle u_k, fg \rangle$ and $\langle E_t, f g\rangle$. Such automorphic triple products are bounded by means of the Watson-Ichino formula, using (sub)convexity bounds, precise control on the Archimedean L-factors, via the Stirling formula, and lower bounds on the special values at $1$ occurring as denominators. More precisely, in case an Eisenstein series arises we use the Rankin-Selberg identity from \Cref{lem:rankin-selberg} and get
\begin{align}
\left\langle u_kg, E_t \right\rangle &= \frac{\rho_k(1)\rho_g(1) \Lambda(\tfrac{1+it}{2}, g \times u_k)}{2\Lambda(1+it)} \label{a} \\ 
\left\langle E_tg, E_\tau \right\rangle &= \frac{\rho_t(1)\rho_g(1) \Lambda(\tfrac{1+it+i\tau}{2}, g)\Lambda(\tfrac{1+i\tau-it}{2}, g)}{2\Lambda(1+it)}.
\end{align} 
In the case only cusp forms occur, we use the Watson-Ichino formula from \Cref{lem:watson} and obtain
\begin{align}
|\langle u_kg, u_j\rangle|^2 & \ll \frac{\Lambda(\tfrac12, u_k \times g \times u_j)}{\Lambda(1, \mathrm{sym}^2 u_k)\Lambda(1, \mathrm{sym}^2 g)\Lambda(1, \mathrm{sym}^2 u_j)} \\
|\langle u_j, f^2 \rangle|^2 & \ll \frac{\Lambda(\tfrac12, u_j \times \mathrm{sym}^2 f)\Lambda(\tfrac12, u_j)}{\Lambda(1, \mathrm{sym}^2 f)^2\Lambda(1, \mathrm{sym}^2 u_j)}. \label{b}
\end{align}

We in particular obtain
\begin{align}
|\left\langle u_kg, E_t \right\rangle|^2 &\ll \frac{ \Lambda(\tfrac12, E_t \times g \times u_k)}{|\Lambda(1+it)|^2\Lambda( 1, \mathrm{sym}^2 g)\Lambda(1, \mathrm{sym}^2 u_k)} \\ 
|\left\langle E_tg, E_\tau \right\rangle|^2 &\ll \frac{ \Lambda(\tfrac12, E_t \times E_\tau \times g)}{|\Lambda(1+it) \Lambda (1+i\tau)|^2 \Lambda(1, \mathrm{sym}^2 g)}.
\end{align} 

\subsection{Bounds on L-factors}

Let $\Gamma_{\C}(s)=2(2\pi)^{-s} \Gamma(s)$. Writing the Gamma factors explicitly, we obtain the following bounds: 
\begin{align}
|\left\langle u_kg, E_t \right\rangle|^2 &\ll \frac{  L(\tfrac12, E_t \times g \times u_k) \prod_{\pm,\pm,\pm} \Gamma_\C(\frac{1\pm it \pm i t_f \pm i t_g}{2})}{|\zeta_K(1+it)|^2 L( 1, \mathrm{sym}^2 g)L(1, \mathrm{sym}^2 u_k) |\Gamma(1+it)|^2\prod_{\pm,\pm} \Gamma_\C(1\pm it_g)\Gamma_\C(1\pm it_k)} \\ 
|\left\langle E_tg, E_\tau \right\rangle|^2 &\ll \frac{ L(\tfrac12, E_t \times E_\tau \times g) \prod_{\pm} \Gamma_\C(\frac{1\pm it \pm i t_g \pm i \tau}{2})} {|\zeta_K(1+it) \zeta_K (1+i\tau)|^2 L(1, \mathrm{sym}^2 g) |\Gamma_\C(1+it)\Gamma_\C(1+i\tau)|^2\prod_{\pm} \Gamma_\C(1\pm \frac{it_g}{2})}.
\end{align} 

In the case of cusp forms we get the following bounds:
\begin{align}
|\langle u_kg, u_j\rangle|^2 & \ll \frac{L(\tfrac12, u_k \times g \times u_j) \prod_{\pm,\pm,\pm}\Gamma_{\C}(\frac{1\pm it_k \pm it_g \pm it_j}{2})}{L(1, \mathrm{sym}^2 u_k)L(1, \mathrm{sym}^2 g)L(1, \mathrm{sym}^2 u_j)\prod_{\pm, v=k,g,j}\Gamma_{\C}(1 \pm it_v)} \\
|\langle u_j, f^2 \rangle|^2 & \ll \frac{L(\tfrac12, u_j \times \mathrm{sym}^2 f)L(\tfrac12, u_j) \prod_{\pm}\Gamma_{\C}(\frac{1\pm it_j}{2})^2 \prod_{\pm, \pm}\Gamma_{\C}(\frac{1\pm it_j}{2} \pm it_f)}{L(1, \mathrm{sym}^2 f)^2L(1, \mathrm{sym}^2 u_j) \prod_{\pm} \Gamma_{\C}(1\pm it_f)^2 \prod_{\pm} \Gamma_{\C}(1\pm it_j)}.  \label{b}
\end{align}

The expressions obtained are similar to the $\mathrm{GL}_2(\Q)$ case of Hua-Huang-Li~\cite{hua_joint_2024}, with $\Gamma_{\C}(s)$ replacing $\Gamma_{\R}(s)=\pi^{-s/2}\Gamma(\tfrac{s}{2})$. As a result, after applying Lemma \ref{lem:stirling}, it is possible to get similar estimates to \cite{hua_joint_2024}, with different polynomial terms. We report the explicit expressions for completeness.

\begin{multline}\label{eq:bound-cusp}
    |\langle u_kg , u_j\rangle \langle u_j, f^2 \rangle|\ll
    \frac{L(1/2, u_j)^{\frac{1}{2}}L(1/2, \sym^2f\times u_j)^{\frac{1}{2}}
    L(1/2, u_k\times g\times u_j)^{\frac{1}{2}}}
    {L(1, \sym^2u_j)L(1, \sym^2f)L(1, \sym^2u_k)^{\frac{1}{2}}L(1, \sym^2g)^{\frac{1}{2}}}
    \\ \times
    \frac{e^{-\frac{\pi}{2} Q_1(t_j; t_f, t_g, t_k)}}
    {t_j t_f t_k^{1/2} t_g^{1/2}},
\end{multline}
\begin{multline}\label{eq:bound-est}
    |\langle E_tg , u_j\rangle \langle u_j, f^2 \rangle|\ll
    \frac{L(1/2, u_j)^{\frac{1}{2}}L(1/2, \sym^2f\times u_j)^{\frac{1}{2}}
    L(1/2, E_t\times g\times u_j)^{\frac{1}{2}}}
    {L(1, \sym^2u_j)L(1, \sym^2f)|\zeta_K(1+it)|L(1, \sym^2g)^{\frac{1}{2}}}
    \\ \times
    \frac{e^{-\frac{\pi}{2} Q_1(t_j; t_f, t_g, t_k)}}
    {t_j t_f (1+|t|) ^{1/2}t_g^{1/2}},
\end{multline}
\begin{multline}\label{eq:bound-ugEEf2}
    |\langle u_kg , E_{\tau}\rangle \langle E_{\tau}, f^2\rangle|\ll
     \frac{|\zeta_K(\frac{1+i\tau}{2})L(1/2+i\tau, \sym^2f)
    L(1/2+i\tau, u_k\times g)|}
    {|\zeta_K(1+i\tau)|^2L(1, \sym^2f)L(1, \sym^2u_k)^{\frac{1}{2}}L(1, \sym^2g)^{\frac{1}{2}}}
    \\ \times
    \frac{e^{-\frac{\pi}{2} Q_1(\tau; t_f, t_g, t_k)}}
    {(1+|\tau|) t_f t_k^{1/2} t_g^{1/2}},
\end{multline}
and
\begin{multline}\label{eq:bound-EgEEf2}
    |\langle E_tg , E_{\tau}\rangle \langle E_{\tau}, f^2\rangle|\ll
    \frac{|\zeta_K(\frac{1+i\tau}{2})L(1/2+i\tau, \sym^2f)\prod_{\pm}
    L(1/2+i\tau\pm it, g)|}
    {|\zeta_K(1+i\tau)|^2|\zeta_K(1+it)|L(1, \sym^2f)L(1, \sym^2g)^{\frac{1}{2}}}
    \\ \times
    \frac{e^{-\frac{\pi}{2} Q_1(\tau; t_f, t_g, t)}}
    {(1+|\tau|)(1+|t|)^{1/2} t_f t_g^{1/2}},
\end{multline}
where
\begin{multline}\label{eq:Q1}
Q_1(t_j; t_f, t_g, t_k)=\Big|\frac{t_j}{2}
+t_f\Big|+\Big|\frac{t_j}{2}
-t_f\Big|
+\frac{|t_j+t_g+t_k|}{2}
+\frac{|t_j+t_g-t_k|}{2}
\\
+\frac{|t_j-t_g+t_k|}{2}
+\frac{|t_j-t_g-t_k|}{2}
-t_j-2t_f-t_k-t_g.
\end{multline}
\subsection{Final Estimates}
We apply the bounds obtained to the study of
 \begin{equation} 
\label{eq:spectral-sum-product-cusp}
\sum_j \langle gu_k, u_j \rangle \langle u_j, f^2 \rangle.
 \end{equation}
The other terms in \eqref{eq:spectral-expansion-3} can be treated in a similar way. The main input comes from the study of the function $Q_1(t_j,t_f,t_g,t_k)$. From \eqref{eq:bound-cusp}, we get that if $Q_1$ is large enough, the exponential contribution will dominate the polynomial terms $P(t_j,t_f, t_g, t_k)$ coming from the convexity bounds available for the central values of the $L$-functions considered, (see~\cite{ik}) and the logarithmic factors coming from the symmetric square $L$-functions (this is equivalent to the non-existence of Siegel zeros, and this is proven for symmetric squares of $\mathrm{GL}(2)$-forms in \cite{hoffstein_siegel_1995}, conditionally on an hypothesis proven in \cite{banks_twisted_1997}). For this reason, the sum in \eqref{eq:spectral-sum-product-cusp} can be truncated to $t_j<2t_f+t_k+t_g$. In fact, in the other case we have $Q_1>2t_f+t_k+t_j$ and it grows linearly in $t_j$, which results in an exponential decay in $t_j$ for the considered terms, from \eqref{eq:bound-cusp}. This implies that the contribution of the second range is negligible or comparable to the first range unconditionally. 

We now split the study in two cases. If $2t_f<t_g-t_g^{\epsilon}$, we obtain that $Q_1(t,t_f,t_j,t_k)>t_g-2t_f-t_g^{o(1)}$, where we used $t_k=t_g^{o(1)}$. Then, an application of the convexity upper bounds for the $L$-functions and the logarithmic lower bounds for the  symmetric square $L$-functions give unconditionally an exponential bound of the form
\begin{equation}
\sum_j \langle gu_k, u_j \rangle \langle u_j, f^2 \rangle \ll e^{-\frac{\pi}{2}(t_g-2t_f-t_g^{o(1)})}.
\end{equation}
If $2t_f>t_g-t_g^{\epsilon}$, we assume GLH. Using the trivial bound $Q_1 \geq 0$ and \eqref{eq:bound-cusp} we obtain
\begin{equation}
\label{eq:GLH-bound}
   \sum_{t_j<2t_f+t_k+t_g} |\langle u_kg , u_j\rangle \langle u_j, f^2 \rangle|\ll
      \sum_{t_j<2t_f+t_k+t_g} \frac{(t_g t_j t_f)^{\delta}}{t_j t_f t_g^{1/2} t_k^{1/2}} \ll (t_f t_g^{1/2})^{-1+\delta'}
\end{equation}
for any $\delta'>\delta>0$, using again the hypothesis $t_k=t_g^{o(1)}$.

This concludes the proof of \Cref{thm}. \qed 

\begin{rem} 
A more accurate study of the conductors of the L-functions involved reveals the need of a strong form of subconvexity to achieve an $o(1)$ error term. In particular, in the numerator of ~\eqref{eq:bound-cusp} we get a degree 16 $L$-function, whose analytic conductor $C$ is bounded by 
\[
C\ll t_j^4 t_f^{12}.
\]
As a consequence, the numerator grows as $(t_j^2 t_f^6)^E$ and a subconvex bound with exponent $E=1/8-\delta$ for the L-functions  is necessary to achieve the result. 
\end{rem}

\begin{rem}
We expect a similar result to hold in the level aspect. The result would follow from the contributions to the Watson-Ichino formula coming from the ramified places. The evaluation of such contributions can be computed explicity as in Lemma \ref{lem:watson}.
\end{rem}
\begin{rem}
The methods of this paper generalize directly to the general case of number fields $F$ of class number one. The contribution given by the different infinite places in the analogue of Section 3.3 can be treated separately. The estimates finally reduce to~$r$ copies of the real case as in \cite{hua_joint_2024} and $s$ copies of the complex case treated in the present article, for $r$ and $s$ respectively the number of real and complex embeddings of~$F$. The expected result is a bound on  the error term of the form $$O_{\psi}((t_f(1+|2t_f-t_g|))^{-r/4+\epsilon}(t_ft_g^{1/2})^{-s+\epsilon}).$$
\end{rem}
%

\textbf{Data availability statement.} Data sharing is not applicable --- the paper describes entirely theoretical research.

\textbf{Conflict of interest statement} The authors declare that they have no conflicts of interest.

\textbf{Acknowledgments} L. M. is Co-Funded by the European Union’s Horizon Europe research and innovation programme under the Marie Sklodowska-Curie Grant Agreement No 101126554. Views and opinions expressed are however those of the authors only and do not necessarily reflect those of the European Union. Neither the European Union nor the granting authority can be held responsible for them.

\begin{figure}[h]
\centering
\includegraphics[scale=0.22]{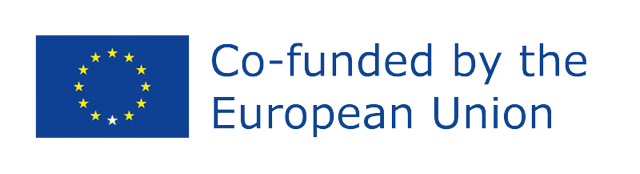}
\end{figure}

\nocite{*}
\bibliographystyle{amsalpha} 
\bibliography{biblio.bib} 

\end{document}